\documentclass[psamsfonts, draft, 12pt]{amsproc}
\usepackage{amssymb,amsmath}
\usepackage{amsthm}
\usepackage{mathrsfs}
\usepackage{enumerate, indentfirst}
\usepackage[fleqn,tbtags]{mathtools}
\usepackage{courier}
\usepackage[a4paper,left=2.5cm,right=2.5cm,top=3cm,bottom=3cm]{geometry}
\usepackage{hyperref}

\usepackage{xr}

\hypersetup{
    citecolor=black,%
    filecolor=black,%
    linkcolor=black,%
    pdffitwindow=true,%
    colorlinks=false,%
    unicode=true,
    }
\newtheoremstyle{mytheorem}%
{5pt}%
{3pt}%
{\itshape}%
{1pt}%
{\bf}%
{.}%
{.5em}%
{}%

\newtheoremstyle{myremark}%
{5pt}%
{3pt}%
{\upshape}%
{1pt}%
{\em}%
{.}%
{.5em}%
{}%

\theoremstyle{mytheorem}
\newtheorem{theorem}{Theorem}[section]

\newtheorem{corollary}[theorem]{Corollary}
\theoremstyle{definition}

 \theoremstyle{myremark}

\makeatletter \renewenvironment{proof}[1][\proofname] {\par\pushQED{\qed}\normalfont\topsep6\p@\@plus6\p@\relax\trivlist\item[\hskip\labelsep\itshape #1\@addpunct{.}]\ignorespaces}{\popQED\endtrivlist\@endpefalse} \makeatother

\renewcommand{\phi}{\varphi}
\renewcommand{\theta}{\vartheta}

\newcommand{\lmu}{\mathscr{L}^2(X,\mathfrak{M},\mu)}

\newcommand{\dom}{\operatorname{dom}}

\newcommand{\ran}{\operatorname{ran}}

\newcommand{\hil}{\mathfrak{H}}

\newcommand{\kil}{\mathfrak{K}}

\DeclarePairedDelimiterX\sip[2]{(}{)}{#1\,\delimsize\vert\,#2}
\DeclarePairedDelimiterX\sipn[2]{(}{)_{\nu}}{#1\,\delimsize\vert\,#2}
\DeclarePairedDelimiterX\sipm[2]{(}{)_{\mu}}{#1\,\delimsize\vert\,#2}
\DeclarePairedDelimiterX\set[2]{\{}{\}}{#1\,\delimsize\vert\,#2}

\newcommand{\ort}[1]{\{#1\}^{\perp}}

\begin{document}
\title{Operators having selfadjoint squares}

\author[Z. Sebesty\'en]{Zolt\'an Sebesty\'en}
\address{Z. Sebesty\'en, Department of Applied Analysis, E\"otv\"os L. University, P\'azm\'any P\'eter s\'et\'any 1/c., Budapest H-1117, Hungary; }
\email{sebesty@cs.elte.hu}

\author[Zs. Tarcsay]{Zsigmond Tarcsay}
\address{Zs. Tarcsay, Department of Applied Analysis, E\"otv\"os L. University, P\'azm\'any P\'eter s\'et\'any 1/c., Budapest H-1117, Hungary; }

\email{tarcsay@cs.elte.hu}


\maketitle

\begin{abstract}
The main goal of this paper is to show that a (not necessarily densely defined or closed) symmetric operator $A$ acting on a real or complex  Hilbert space is selfadjoint exactly when $I+A^2$  is a full range operator.
\end{abstract}

\section{Introduction}

If $T$ is a densely defined closed operator between two Hilbert spaces, $\hil$ and $\kil$, a classical theorem due to John von Neumann \cite{vonNeumann} states that $I+T^*T$ is selfadjoint operator with full range. As an immediate consequence of that result one obtains also that the square of a selfadjoint operator, say $A$, is selfadjoint as well,  furthermore, that $I+A^2$ is surjective. If the underlying Hilbert space $\hil$ is complex, by employing the classical theory of deficiency indices, also due to von Neumann \cite{vonNeumann1930}, we conclude that the converse of the latter statement is also true. Precisely, if $A$ is a densely defined symmetric operator in a complex Hilbert space $\hil$ such that $I+A^2$ is surjective, then the original operator $A$ must be selfadjoint. Indeed, according to the factorizations
\begin{equation}\label{factorization}
    A^2+I=(A+iI)(A-iI)=(A-iI)(A+iI),
\end{equation}
it is seen readily that both $A\pm iI$ must be onto, and therefore that $A$ is selfadjoint.

Factorization \eqref{factorization} cannot be used, of course, when the underlying Hilbert space $\hil$ is real. Furthermore, if the symmetric operator is not densely defined, the theory of deficiency indices is again unapplicable, even if $\hil$ is complex.

The main purpose of this note is to prove that the following characterization of selfadjointness holds, be the underlying Hilbert space real or complex: a symmetric operator $A$ on a (real or complex) Hilbert space is selfadjoint if and only if $I+A^2$ is surjective. Observe also that the symmetric operator under consideration is not assumed to be densely defined a priori. On the contrary, densely definedness is also a direct consequence of our other assumptions.

\section{Operators having selfadjoint squares}
 Recall that an operator $A$ defined in a Hilbert space $\hil$ is said to be symmetric if
 \begin{equation*}
    \sip{Ax}{y}=\sip{x}{Ay}, \qquad x,y\in\dom A,
 \end{equation*}
 and skew-symmetric if
 \begin{equation*}
    \sip{Ax}{y}=-\sip{x}{Ay}, \qquad x,y\in\dom A.
 \end{equation*}
 If $A$ is densely defined in addition then the symmetry (resp., skew-symmetry) of $A$ means that $A\subseteq A^*$ (resp., $A\subseteq -A^*$). Furthermore, a densely defined operator $A$ is said to be selfadjoint (resp., skew-adjoint) if $A=A^*$ (resp., $A=-A^*$). Note also immediately that each selfadjoint (resp., skew-adjoint) operator is closed.

 Our first result is a characterization of the skew-adjointness of an operator in terms of its square:
\begin{theorem}\label{T:maintheorem1}
    Let $\hil$ be real or complex Hilbert space and $A:\hil\to \hil$ a skew-symmetric linear operator, whose domain $\dom A$ is not assumed to be dense. The following statements are equivalent:
    \begin{enumerate}[\upshape (i)]
      \item $A$ is densely defined and skew-adjoint;
      \item $-A^2$ is a (positive) selfadjoint operator;
      \item $I-A^2$ is a  full range operator, i.e. $\ran (I-A^2)=\hil$.
    \end{enumerate}
\end{theorem}
\begin{proof}
    If $A$ is skew-adjoint, then clearly, $A$ is closed, and the following identity
    \begin{equation*}
        -A^2=A^*A
    \end{equation*}
    shows statement (ii), thanks to von Neumann's classical theorem. By assuming (ii), the operator $I-A^2$ is positive and selfadjoint, and  bounded below (by one), therefore its range is dense, and closed in $\hil$, that is $\ran (I-A^2)=\hil.$ Assume finally that the symmetric operator $I-A^2$ is of full range. Then it is densely defined and positive selfadjoint, as we see at once. First, $\dom (I-A^2)$ is dense, for if $y$ is from $\ort{\dom (I-A^2)}=\ort{\dom A^2}$, then one takes into account that $y=(I-A^2)z$ for some $z\in\dom A^2.$ We have at once for each $x$ from $\dom (I-A^2)$ that
    \begin{equation*}
        0=\sip{x}{(I-A^2)z}=\sip{(I-A^2)x}{z}.
    \end{equation*}
    Therefore, $z$ belongs to $\ort{\ran(I-A^2)}=\{0\}$ by assumption, thus $y=0$, as claimed.

    One more consequence is that $A$ is densely defined skew-symmetric operator, thus fulfilling the following identity:
    \begin{equation*}
        A\subset-A^*.
    \end{equation*}
    To prove statement (i) one checks only that $\dom A^*\subseteq\dom A$. Let now $y\in\dom A^*$ and take some $z\in\dom A^2$ by assumption so that
    \begin{equation*}
        y-A^*y=(I-A^2)z=(I+A)(I-A)z.
    \end{equation*}
    Then we have, since $I+A\subset I-A^*$, that
    \begin{align*}
        (y-(I-A)z)\in\ker(I-A^*)&=\ker(I-A)^*=\ort{\ran(I-A)}\\ &\subset\ort{\ran (I-A^2)}=\{0\}.
    \end{align*}
    This means just that $y=(I-A)z\in\dom A$, as it is claimed.
\end{proof}

The main result of our paper is the following statement:
\begin{theorem}\label{T:maintheorem2}
    Let $\hil$ be real or complex Hilbert space and $A:\hil\to\hil$ a symmetric operator whose domain is not assumed to be dense subspace in $\hil$. The following assertions are equivalent:
    \begin{enumerate}[\upshape (i)]
      \item $A$ is densely defined and selfadjoint operator;
      \item $A^2$ is a positive selfadjoint operator;
      \item $I+A^2$ is a full range operator, i.e. $\ran (I+A^2)=\hil$.
    \end{enumerate}
\end{theorem}
\begin{proof}
    If $A$ is assumed to be selfadjoint, then $A^2=A^*A$ is positive selfadjoint operator in virtue of von Neumann's classical theorem. Therefore, (i) implies (ii). Statement (ii) also clearly implies (iii) as $(I+A^2)$ is positive selfadjoint and bounded below (by one) operator whose range and closed as well, therefore is the whole space $\hil$. It remains to prove implication (i)$\Rightarrow$(ii). First of all, $A^2$ is densely defined: for if $y$ is from $\ort{\dom A^2}$, then, since $y=(I+A^2)z$ for some $z$ from $\dom A^2$, and at the same time for each $x$ from $\dom A^2$
    \begin{align*}
        0=\sip{x}{y}=\sip{x}{(I+A^2)z}=\sip{(I+A^2)x}{z}
    \end{align*}
    holds true. This means, of course, that $z$ is from $\ort{\ran (I+A^2)}=\{0\}$, and therefore that $y=0$, indeed.
    One more consequence is that $\dom A$ is dense as well, and thus
    \begin{equation*}
        A\subset A^*,
    \end{equation*}
    by our assumption on the symmetricity of $A$.

    The last step in to check that $A$ is selfadjoint is that $\dom A^*\subseteq\dom A$ as follows. Take $z\in\dom A^*$, then for some $x$ and $y$ from $\dom A^2$ we have that
    \begin{equation*}
        A^*z=(I+A^2)x\qquad \textrm{and}\qquad -z=(I+A^2)y.
    \end{equation*}
    This means at the same time that
    \begin{align*}
    \left\{
      \begin{array}{l}
        -z=A(x+Ay)-(Ax-y), \\
        A^*z=A(Ax-y)+(x+Ay),
      \end{array}
    \right.
    \end{align*}
    and consequently, since $Ax-y\in\dom A\subseteq\dom A^*$, that  $(z-(Ax-y))\in\dom A^*$ and
    \begin{align*}
        A^*(z-(Ax-y))=A^*z-A^*(Ax-y)=A^*z-A(Ax-y)=x+Ay,
    \end{align*}
    and as well that
    \begin{equation*}
        0=-A^*z+A^*z=A^*A(Ax+y)+(x+Ay).
    \end{equation*}
    As a consequence we finally have that
    \begin{align*}
        0&=\sip{A^*A(Ax+y)}{Ax+y}+\sip{x+Ay}{x+Ay}\\
         &=\|A(Ax+y)\|^2+\|x+Ay\|^2.
    \end{align*}
    Therefore $x+Ay=0$, so that $z=Ax-y\in\dom A$, indeed. The proof is complete.
\end{proof}
Another characterization of selfadjoint and skew-adjoint operators involving the ranges of $I\pm A^2$ is given in the next corollary:
\begin{corollary}
    Let $A$ be a densely defined symmetric (resp., skew-symmetric) operator in the real or complex Hilbert space $\hil$. Then the following are equivalent:
    \begin{enumerate}[\upshape (i)]
      \item $A$ is selfadjoint (resp., skew-adjoint);
      \item $\dom A^*\subseteq \ran (I+A^2)$ and $\ran A^{**}\subseteq \ran (I+A^2)$ (resp., $\dom A^*\subseteq \ran (I-A^2)$ and $\ran A^{**}\subseteq \ran (I-A^2)$);
      \item $\dom A^{**}\subseteq \ran (I+A^2)$ and $\ran A^{*}\subseteq \ran (I+A^2)$ (resp., $\dom A^{**}\subseteq \ran (I-A^2)$ and $\ran A^{*}\subseteq \ran (I-A^2)$).
    \end{enumerate}
\end{corollary}
\begin{proof}
    If $A$ is selfadjoint (resp., skew-adjoint), then Theorem \ref{T:maintheorem2} (resp., Theorem \ref{T:maintheorem1}) implies that $I+A^2$ (resp., $I-A^2$) has full range. Thus (i) implies either of (ii) and (iii). Conversely, for a densely defined closable operator $T$, acting between Hilbert spaces $\hil$ and $\kil$, we have the following well known identities:
    \begin{equation*}
        \dom T^{**}+\ran T^*=\hil,\qquad \dom T^*+\ran T^{**}=\kil.
    \end{equation*}
    Hence, each of (ii) and (iii) implies that $\ran (I+A^2)=\hil$ (resp., $\ran (I-A^2)=\hil$). Due to Theorem \ref{T:maintheorem2} (resp., Theorem \ref{T:maintheorem1}) this means that $A$ is selfadjoint (resp., skew-adjoint).
\end{proof}
\begin{corollary}
    Let $(X,\mathfrak{M},\mu)$ be a measure space and $f$ be any real valued measurable function of $X$. The multiplication operator $A$ by $f$ on the (real or complex) Hilbert space $\lmu$ with maximal domain,
    \begin{equation*}
        \dom A=\set{g\in\lmu}{f\cdot g\in\lmu},
    \end{equation*}
    is selfadjoint.
\end{corollary}
\begin{proof}
    It is readily seen that $A$ is a symmetric operator. For a given $g\in\lmu$, one obtains at once that $h=\dfrac{g}{1+f^2}$ belongs to $\dom A^2$ so that $(I+A^2)h=g$. That means precisely that $I+A^2$ is of full range, and therefore, in account of Theorem \ref{T:maintheorem2}, that $A$ is selfadjoint.
\end{proof}
\begin{theorem}
    Let $\hil$ be a real or complex Hilbert space, and $A:\hil\to\hil$ be a positive symmetric operator, not assumed to be densely defined. The following statements are equivalent:
    \begin{enumerate}[\upshape (i)]
      \item $A$ is selfadjoint;
      \item $I+A$ is of full range, i.e. $\ran(I+A)=\hil$.
    \end{enumerate}
\end{theorem}
\begin{proof}
    It is clear that (i) implies (ii): $I+A$ is bounded below (by one) closed operator, therefore its range is dense and closed, i.e. $\ran (I+A)=\hil.$ Conversely, $I+A$ to be a full range operator. First of all $A$ is densely defined: for if $y\in\ort{\dom A}$ then, since $y=(I+A)z$ for some $z\in\dom A$ and then for each $x$ from $\dom A$ we have that
    \begin{equation*}
        0=\sip{x}{y}=\sip{x}{(I+A)z}=\sip{(I+A)x}{z}.
    \end{equation*}
    Therefore, $z\in\ort{\ran(I+A)}=\{0\}$, i.e. $z=0$ and then $y=0$ as claimed.

    Next we have at once that
    \begin{equation*}
        (I+A)\subset(I+A)^*=(I+A^*),
    \end{equation*}
    so that $A^*=A$ is the same as $(I+A)^*=I+A$. If $y\in\dom A^*$ then we see that for some $z\in\dom A$
    \begin{equation*}
        y+A^*y=(I+A)z=(I+A^*)z,
    \end{equation*}
    and therefore
    \begin{equation*}
        (y-z)\in\ker(I+A^*)=\ort{\ran(I+A)}=\{0\}.
    \end{equation*}
    Consequently, $y=z\in\dom A$, as it is claimed.
\end{proof}
\begin{corollary}
    Let $\hil$ and $\kil$ be real or complex Hilbert spaces, $T:\hil\to\kil$ be densely defined linear operator. Then $T^*T$ is positive selfadjoint if and only if $\ran(I+T^*T)=\hil$. If $T$ is closed, then $T^*T$ is positive selfadjoint operator on $\hil$.
\end{corollary}
\begin{proof}
    We should only check that if $T$ is closed then $\ran (I+T^*T)=\hil$. Of course, this is the case when the two closed subspaces are orthogonal complements on $\hil\times\kil$:
    \begin{equation*}
        \set[\big]{(x,Tx)}{x\in\dom T}\qquad\textrm{and}\qquad\set[\big]{(-T^*z,z)}{z\in\dom T^*}.
    \end{equation*}
    Therefore, for each $y\in\hil$ we find $x\in\dom T$ and $z\in\dom T^*$ such that
     \begin{equation*}
        y=x-T^*z\qquad\textrm{and}\qquad0=Tx+z.
    \end{equation*}
    Consequently, $-z=Tx\in\dom T^*$ and $-T^*z=T^*Tx$ so that
    \begin{equation*}
        y=x+T^*Tx\in\ran (I+T^*T),
    \end{equation*}
    as desired.
\end{proof}

\end{document}